\documentclass[10pt]{article}
\usepackage{color}
\usepackage{amssymb}
\usepackage{amsthm,array,amssymb,amscd,amsfonts,latexsym, url}
\usepackage{amsmath}

\newtheorem{theo}{Theorem}
\newtheorem{prop}[theo]{Proposition}

\newtheorem{lemm}[theo]{Lemma}
\newtheorem{coro}[theo]{Corollary}
\newtheorem{rema}[theo]{Remark}

\newtheorem{conj}[theo]{Conjecture}

\date{}
\title{Segre classes of tautological bundles on Hilbert schemes of surfaces}
\author{Claire Voisin\footnote{Coll\`ege de France  and ETH-ITS}}

\begin{document}
\maketitle
\begin{abstract} We first give an alternative proof,  based on a simple geometric argument, of a result of Marian, Oprea and Pandharipande on top Segre classes of the tautological bundles  on Hilbert schemes of $K3$ surfaces equipped with
 a line bundle. We then turn to the blow-up of $K3$ surface at one point and establish vanishing results for the corresponding  top Segre classes in  a certain range. This  determines, at least theoretically, all top Segre classes of tautological bundles for any pair $(\Sigma,H),\,H\in {\rm Pic}\,\Sigma$.
\end{abstract}

{\bf Classification}. {14N10 (primary),  14J99 (secondary). }

\section{Introduction\label{secintro}}
Let $S$ be  a smooth projective (or compact complex)  surface. The Hilbert scheme $S^{[k]}$ is smooth projective (or compact complex)
of dimension $2k$. For any line bundle
$H$ on $S$, we get an associated vector bundle $\mathcal{H}_{[k]}$ on $S^{[k]}$, whose fiber
at a point $[Z]\in S^{[k]}$ is the vector space $H^0(H_{\mid Z})$. If $S$ is a
$K3$ surface and $c_1(H)^2=2g-2$, we denote
$$s_{k,g}:=\int_{S^{[k]}}s_{2k}(\mathcal{H}_{[k]}).$$
This is indeed a number which depends only on $k$ and $g$ (see Theorem \ref{theoegl}).
The following result is proved in \cite{MOP}:
\begin{theo}\label{theopandha} One has $ s_{k,g}=2^k\binom{g-2k+1}{k}$.
\end{theo}
Here  the binomial coefficient is defined for $k\geq 0$. It is always $1$ for $k=0$ and
the formula for $\binom{n}{k}$ for any $n$ is
$$ \binom{n}{k}=\frac{n(n-1)\dots (n-k+1)}{k!}.$$
In particular, we have
$\binom{n}{k}=0$ if $n\geq 0$ and $n<k$. The theorem above thus
gives in particular the vanishing
\begin{eqnarray} \label{eqvanintrolaz} s_{k,g}=0 \,\,{\rm when}\,\,
g-2k+1\geq 0\,\,{ \rm and }\,\,k>g-2k+1.
\label{eqvan}
\end{eqnarray}
 The proof of this vanishing statement in \cite{MOP} is rather involved and we
are going to give in Section \ref{secgeovan} a direct geometric proof of (\ref{eqvanintrolaz}),  based on a small improvement of Lazarsfeld's arguments in \cite{laz}.

We will then show how the vanishing (\ref{eqvan}), even only in the smaller range $g=2k-1,\,g=2k$,
implies Theorem \ref{theopandha}. We simply use for this the following result which is due to Tikhomirov
\cite{tikho} (see also  Ellingsrud-G\"ottsche-Lehn \cite{egl} and Lehn
\cite{lehn} for related statements) :
\begin{theo} \label{theoegl} The Segre numbers
$\int_{S^{[k]}}s_{2k}(H_{[k]})$ for a projective surface $S$ equipped with a line bundle
$H$ depend only on the four numbers
$$\pi=H\cdot  K_S,\,d=H^2,\,\kappa=K_S^2,\,e=c_2(S).$$
\end{theo}
We will denote these Segre numbers $s_{k,d,\pi,\kappa,e}$.
It follows from Theorem \ref{theoegl} that the numbers
$s_{k,g}$ can be computed as well by considering a surface
$\Sigma$ which is the disjoint union of a $K3$ surface  $S'$, equiped with a line bundle $H'$
of self intersection $2(g-1)-2$, and an abelian surface $A$ equiped with a line bundle
$\theta$ with $\theta^2=2$. We will show in Section
\ref{sec2} that the formula obtained by this observation (this is a particular case
 of (\ref{eqsumprod}) below), combined with the vanishing result
  (\ref{eqvan}), uniquely determine the numbers $b_k:=\int_{A^{[k]}}s_{2k}(\theta_{[2k]})$ and finally the numbers
$s_{k,g}$ for all $k,\,g$, knowing that $s_{1,g}=2g-2$, $b_0=1,\,b_1=2$.

In Section \ref{secgeovan}, we will establish similar vanishing results
for a $K3$ surface $S$
blown-up at one point. Let $\widetilde{S}$ be such a surface
and let $H=\tau^*L(-lE)$ with  $2g-2=L^2$, where $L$ generates ${\rm Pic}\,S$.
\begin{theo}  \label{theoannuintro}
For $k\geq 2$, one has the following vanishing for the Segre numbers $\tilde{s}_{k,g,l}:=\int_{\widetilde{S}^{[k]}}s_{2k}(\mathcal{H}_{[k]})$:
\begin{eqnarray}\label{eqvanconjintro}\tilde{s}_{k,g,l}=0\,\,{\rm for}\,\,k=l,\,l+1\,\,{\rm and}\,\,g-\frac{l(l+1)}{2}=3k-2.
\end{eqnarray}

\end{theo}
We will also prove that these vanishing statements  together with Theorem
\ref{theopandha} determine all Segre numbers $s_k(d,\pi,\kappa,e)$.
We use for this the following complement to  Theorem
\ref{theoegl}, (see \cite{lehn}, \cite{egl},) obtained by observing that
the Hilbert scheme
$S^{[k]}$ of a disjoint union $S_1\sqcup S_2$ is the disjoint union
for $l=0,\ldots,k$, of $S_1^{[l]}\times S_2^{[k-l]}$, while all the data
$d,\,\pi,\,\kappa,\,c_2$ for the pairs $(\Sigma,H)$ are additive under disjoint unions
$(S,L)=(S_1,L_1)\sqcup (S_2,L_2)$:
\begin{lemm}\label{theosumprod} With the notation
$s_{d,\pi,\kappa,e}(z)=\sum_k s_{k,d,\pi,\kappa,e}z^k$
\begin{eqnarray}
\label{eqsumprod} s_{d,\pi,\kappa,e}(z)=s_{d_1,\pi_1,\kappa_1,e_1}(z)s_{d_2,\pi_2,\kappa_2,e_2}(z)
\end{eqnarray}
with $d=d_1+d_2,\,\pi=\pi_1+\pi_2$ etc.
\end{lemm}

To conclude this introduction, we mention Lehn's conjecture \cite[Conjecture 4.9]{lehn}:
\begin{conj}
\label{conjlehn} One has
\begin{eqnarray}
\label{eqcongserie} s_{d,\pi,\kappa,e}(z)=\frac{(1-w)^a(1-2w)^b}{(1-6w+6w^2)^c},
\end{eqnarray}
where $a=\pi-2\kappa,\,b=d-2\pi+\kappa+3\chi,\,c=\frac{d-\pi}{2}+\chi$, $\chi=\frac{\kappa+e}{12}$
and the variable $w$ is related to $z$ by
$$z=\frac{w(1-w)(1-2w)^4}{(1-6w+6w^2)^3}.$$
\end{conj}
This conjecture is proved  in \cite{MOP} for $K3$ and more generally $K$-trivial surfaces, that is for $\kappa=\pi=0$. Although we were not able to prove it in general,
our results imply the following:
\begin{coro} \label{coropourlehnintro} Lehn's conjecture is equivalent to the fact that the
development in power series of $z$ of the Lehn function
$f_{d,\pi,\kappa,e}(z)$
defined as the right hand side in (\ref{eqcongserie})
has vanishing
Taylor coefficient of order $k$ for $e=25,\,\kappa=-1$ and
$(d,\,\pi)=(7(k-1),k-1)$ or $(d,\,\pi)=(7(k-1)+1,k)$
\end{coro}

\vspace{0.5cm}

Shortly after this paper was written,  Marian-Oprea-Pandharipande
(see \cite{MPRlehn}) and Szenes-Vergne independently  were able to check that the Lehn function satisfies the vanishing properties
stated in Corollary \ref{coropourlehnintro}, thus completing the proof of Lehn's conjecture.

Let us mention the following intriguing question: Lehn's conjecture (now a theorem) singles out
the class of pairs $(S,H)$ with the following
numerical properties:

\begin{eqnarray}\label{eqsirfsingle} H^2=0,\,H\cdot K_S=2K^2=2\chi(\mathcal{O}_S).
\end{eqnarray}

These conditions are indeed equivalent to the vanishing of the exponents $a,\,b$ and $c$ above, so that for these pairs, one has the vanishing $s_{2k}(\mathcal{H}_{[k]})=0$. It would be
nice to have a geometric proof of this.

\vspace{0.5cm}

{\bf Thanks.}  I thank Rahul Pandharipande for discussions and in particular for suggesting, after I had given a geometric proof of
 the vanishings (\ref{eqvan}) on $K3$ surfaces, to look at surfaces other than $K3$'s.
This work has been done during my stay at ETH-ITS. I acknowledge the support of
Dr. Max R\"ossler, the Walter Haefner Foundation and the ETH Zurich
Foundation.
\section{Geometric vanishing \label{secgeovan}}
Let $S$ be a $K3$ surface with ${\rm Pic}\,S=\mathbb{Z}H$, where
$H$ is an ample line bundle of self-intersection $2g-2$.
We give in this section a geometric proof of the vanishing result (\ref{eqvan}) proved in
\cite{MOP}.
\begin{prop}\label{provan} The Segre classes $s_{2k}(H_{[k]})$ vanish in the range
\begin{eqnarray} \label{eqrange} 3k-1>g>2k-2.
\end{eqnarray}
In particular, $s_{k,2k}=0$ and $s_{k,2k-1}=0$ when $k\geq 2$.
\end{prop}
\begin{proof} Sections of $H$ provide sections of $\mathcal{H}_{[k]}$, or equivalently
of the line bundle
$\mathcal{O}_{\mathbf{P}(\mathcal{H}_{[k]}^*)}(1)$. In  fact, all  sections
 of $\mathcal{H}_{[k]}$ come from $H^0(S,H)$. As we are on a $K3$ surface, $H^0(S,H)$ has dimension
$g+1$. We thus have a rational map
$\phi:\mathbf{P}(\mathcal{H}_{[k]}^*)\dashrightarrow\mathbf{P}^g$ such that $\phi^*\mathcal{O}_{\mathbf{P}^g}(1)=\mathcal{O}_{\mathbf{P}(\mathcal{H}_{[k]}^*)}(1)$.
The top Segre class of $\mathcal{H}_{[k]}^*$ (or $\mathcal{H}_{[k]}$) is the top self-intersection
of $c_1(\mathcal{O}_{\mathbf{P}(\mathcal{H}_{[k]}^*)}(1))$ on
$\mathbf{P}(\mathcal{H}_{[k]}^*)$.
We observe that the first inequality in (\ref{eqrange}) says that ${\rm dim}\,\mathbf{P}(\mathcal{H}_{[k]}^*)>{\rm dim}\,\mathbf{P}^g$, so the proposition is a consequence of
the following lemma which is a mild  generalization of Lazarsfeld's result
in \cite{laz}, saying that
smooth curves in $\mid H\mid$ are Brill-Noether generic:
\begin{lemm}\label{lelaz} If $g>2k-2$, the vector bundle $\mathcal{H}_{[k]}$ is generated by the sections
coming from $H^0(S,H)$.
\end{lemm}
Indeed, this last statement says that the  rational map $\phi$ is actually a morphism so that
the top self-intersection of a line bundle pulled-back via $\phi$ is $0$.
\end{proof}
\begin{proof}[Proof of Lemma \ref{lelaz}] The proof is by contradiction. It is obtained by applying Lazarsfeld's arguments in
\cite{laz}.
For convenience of the reader and because Lazarsfeld considers only subschemes supported on smooth curves, we give  the complete argument: If $z\in S^{[k]}$ is a point
such that $H^0(S,H)\rightarrow  \mathcal{H}_{[k],z}$ is not surjective, $z$ corresponds to a length $k$ subscheme  $Z\subset S$  such that the restriction map
$H^0(S,H)\rightarrow H^0(H_{\mid Z})$ is not surjective, hence $H^1(S,\mathcal{I}_Z(H))\not=0$.
By Serre  duality, we thus have a nonzero class $e\in {\rm Ext}^1(\mathcal{I}_Z,H^{-1})$, which provides a
torsion free rank $2$ sheaf  $\mathcal{E}$ fitting into an exact sequence
\begin{eqnarray}\label{eqexactlaz}  0\rightarrow H^{-1}\rightarrow \mathcal{E}\rightarrow \mathcal{I}_Z\rightarrow 0.
\end{eqnarray}

Note that  the original  Lazarsfeld argument deals with all
subschemes which are locally complete intersection, for which $\mathcal{E}$ is locally free (assuming $k$ is minimal).
We have $c_1(\mathcal{E})=H^{-1}$ and $c_2(\mathcal{E})=k$. It thus follows that
$$\chi(\mathcal{E},\,\mathcal{E}):=h^0(End(\mathcal{E}))-{\rm dim}\,Ext^1(\mathcal{E},\mathcal{E})+
{\rm dim}\,Ext^2(\mathcal{E},\mathcal{E})=4\chi(\mathcal{O}_S)+c_1(\mathcal{E})^2-4c_2(\mathcal{E})
$$
$$=8+2g-2-4k.$$
The second inequality in (\ref{eqrange}) thus gives
$$\chi(\mathcal{E},\,\mathcal{E})
>2.$$
We thus conclude (applying Serre  duality
showing that
${\rm dim}\,Ext^2(\mathcal{E},\mathcal{E})=h^0(End(\mathcal{E}))$) that $\mathcal{E}$ has an endomorphism $f:\mathcal{E}\rightarrow \mathcal{E}$
which is not proportional to the identity, hence can be assumed to be of generic  rank $ 1$.
Let $B$ be the line bundle defined as $\mathcal{F}^{**}$ where $\mathcal{F}$ is the saturation
of ${\rm Im}\,f$ in $\mathcal{E}$. The line bundle $B$ must be a power of $H$. The non-split exact sequence (\ref{eqexactlaz}) shows that
${\rm Hom}\,(\mathcal{E},H^{-1})=0$ since  the exact sequence (\ref{eqexactlaz}) is not split, so $B$ must be trivial or  a positive power of
$H$.
It follows that  $ \mathcal{F}$ is equal to
$H^{\otimes k}\otimes \mathcal{I}_W$ for some $k\geq 0$ and for some $0$-dimensional subscheme
 $W\subset Z$ (which can appear only where $\mathcal{E}$ is not locally free).  As $H^{\otimes k}\otimes \mathcal{I}_W$ is not contained
 in $H^{-1}$, it must map nontrivially   to
 $\mathcal{I}_Z$ via $f: \mathcal{E}\rightarrow \mathcal{I}_Z$, so that finally $k=0$ and
 $\mathcal{I}_W\subset \mathcal{I}_Z$.
As $\mathcal{I}_Z\subset \mathcal{I}_W$ and
${\rm End}(\mathcal{I}_Z)=\mathbb{C}Id$, we conclude that in fact
$f$ induces an isomorphism $\mathcal{I}_W\cong \mathcal{I}_Z$ and
the sequence (\ref{eqexactlaz}) is split, which is a contradiction.
\end{proof}

We note for later reference the following simple fact on which  the proof of Proposition
\ref{provan} rests. We will say that $H$ is $k$-ample if
$\mathcal{H}_{[k]}$ is generated by its global sections. $1$-ample means that
$H$ is generated by sections, and $2$-ample means that $H$ is very ample.
\begin{lemm} \label{lefact} Let $\Sigma$ be a surface, $H$ a line bundle on $\Sigma$. Assume
that $H$ is $k$-ample and $h^0(\Sigma,H)< 3k$. Then
$s_{2k}(\mathcal{H}_{[k]})=0$.
\end{lemm}
\section{Proof of Theorem \ref{theopandha} \label{sec2}}
We are going to prove here Theorem \ref{theopandha} for $2g-2\geq 0$, i.e. $g\geq 1$, by induction on $g$. The case where
$g$ is nonpositive works similarly, by induction on $-g$.
Let $S'$ be a $K3$ surface equiped with a line bundle $H'$ such that $c_1(H')^2=2(g-1)-2$. Let $A$
be an abelian surface with a principal polarization
$\theta$, so that $c_1(\theta)^2=2$.
The surface $\Sigma=S'\sqcup A$ equiped with the line bundle
$H_\Sigma$ which is equal to $H'$ on $S'$ and $\theta$ on $A$, has the same characteristic numbers
as our original pair $(S,H)$ where $S$ is a $K3$ surface, and $H$ is  a polarization with self-intersection $2g-2$.
On the other hand, $\Sigma^{[k]}$ is the disjoint union
$$\Sigma^{[k]}=\sqcup_{l=0}^{l=k} {S'}^{[k-l]}\times A^{[l]},$$
and on each summand ${S'}^{[k-l]}\times A^{[l]}$, the vector bundle
$H_{\Sigma,[k]}$ equals
$pr_1^*H'_{[k-l]}\oplus pr_2^* \theta_{[l]}$.
We thus conclude that we have the following formula, where $b_l:=\int_{A^{[l]}}s_{2l}(\theta_{[2l]})$ (this is a particular case of (\ref{eqsumprod})):
\begin{eqnarray}\label{eqrec} s_{k,g}=\sum_{l=0}^{l=k} b_l s_{k-l,g-1}.
\end{eqnarray}
\begin{coro} \label{coroutile} The numbers $s_{k,g}$ for $g\geq 1$  are fully determined by the numbers $b_l,\,0\leq l\leq k$ and the numbers
$s_{l,1},\,l\leq k,\,s_{1,g},\,g\geq 1$ (or $s_{0,g}$).
\end{coro}

\begin{rema} \label{remark} {\rm We have $b_0=1$, $b_1=2$, and similarly $s_{0,g}=1,\,s_{1,g}=2g-2$.}
\end{rema}
\begin{lemm}
Suppose that the numbers $b_l,\,0\leq l\leq k-1$ and the numbers
$s_{l,1},\,0\leq l\leq k-1$ are given, with $b_0=1,\,b_1=2$. Then the numbers $s_{k,1}$ and $b_k$ are determined
by the condition $b_0=1,\,b_1=2$, equation (\ref{eqrec}), and the vanishing equations
\begin{eqnarray}\label{eqvanseulmentdeux} s_{k,2k}=0,\,s_{k,2k-1}=0
\end{eqnarray}
for $k\geq 2$
proved in Proposition \ref{provan}.
\end{lemm}
\begin{proof} Indeed, by Corollary \ref{coroutile},  all the numbers
$s_{l,g'}$ for $ g'\leq g-1$ and $l\leq k-1$ are determined by $b_l,\,0\leq l\leq k-1$ and
$s_{l,1},\,0\leq l\leq k-1$.
We thus can write (\ref{eqrec}) as
$$s_{k,g}=s_{k,g-1}+(\ldots)+ b_k,$$
$$s_{k,g-1}=s_{k,g-2}+(\ldots)+b_k,$$
$$\ldots$$
where the expressions $(\ldots)$ in the middle are determined by $b_l,\,0\leq l\leq k-1$ and
$s_{l,1},\,0\leq l\leq k-1$.
Combining these equations, we get
\begin{eqnarray} s_{k,2k}=s_{k,1}+(\ldots)+ (2k-1)b_k\\
\nonumber
s_{k,2k-1}=s_{k,1}+(\ldots)+(2k-2)b_k,
\end{eqnarray}
hence we can see the equations $s_{k,2k}=0,\,s_{k,2k-1}=0$ as
 a system of two affine equations in the two variables $s_{k,1}$ and $b_k$, whose
linear part is invertible and the constants are determined by $b_l,\,0\leq l\leq k-1$ and
$s_{l,1},\,0\leq l\leq k-1$. The numbers $s_{k,1}$ and $b_k$ are thus uniquely determined by these equations and the numbers  $b_l,\,0\leq l\leq k-1$ and
$s_{l,1},\,0\leq l\leq k-1$.
\end{proof}
\begin{coro}\label{corounique}  There exist unique sequences of  numbers
$s_{k,g},\,k\geq 0,\,g\geq1$ and $b_l,\,l\geq 0$ satisfying:

\begin{enumerate}
\item \label{i} $b_0=1$, $b_1=2$,
\item \label{ia} $s_{0,g}=1$, $s_{1,g}=2g-2$,
\item \label{ii} $s_{k,2k}=0,\,s_{k,2k-1}=0$ for $k\geq 2$.
\item \label{iii} $s_{k,g}=\sum_{l=0}^{l=k} b_l s_{k-l,g-1}$.
\end{enumerate}
\end{coro}
\begin{proof}[Proof of Theorem \ref{theopandha}]
The numbers $s'_{k,g}:=2^k\binom{g-2k+1}{k}$ satisfy the vanishings $s'_{k,2k}=0,\,s'_{k,2k-1}=0$ for $k\geq 2$, that is, condition \ref{ii} of Corollary \ref{corounique}. They also satisfy the condition
$s'_{1,g}=2g-2$, that is, condition \ref{ia} of Corollary \ref{corounique}.
In order to show that $s_{k,g}=s'_{k,g}$, it suffices by Corollary \ref{corounique} to show
that they also satisfy condition \ref{iii} for adequate numbers $b'_l$, which is proved in the following
Lemma \ref{lecombelow}.
\end{proof}
\begin{lemm} \label{lecombelow} There exist numbers $b'_l,\,l\geq0$ with $b'_0=1,\,b'_1=2$ such that
for any $g\geq 1$
\begin{eqnarray}\label{eqconcl} s'_{k,g}=\sum_{l=0}^{k} b'_l s'_{k-l,g-1}.
\end{eqnarray}
\end{lemm}
\begin{proof} We observe that $s'_{k,g}$ is, as a function of $g$, a polynomial of degree exactly $k$, with leading coefficient $2^k$. Hence the $s'_{l,g}$ for $0\leq l\leq k-1$ form a basis of the space of polynomials of degree $k-1$, and for $k$ fixed, there exist uniquely defined  numbers $b'_{l,k}$, $l=0,\ldots, k$, with $b'_{0,k}=1$, such that for any $g$:
\begin{eqnarray}\label{eqconcldepk} s'_{k,g}=\sum_{l=0}^{k} b'_{l,k} s'_{k-l,g-1}.
\end{eqnarray}
Let us prove that $b'_{l,k}=b'_{l,k-1}$ for $l\leq k-1$. We have
$$ \binom{g-2k+1}{k}=\binom{g-2k}{k}+\binom{g-2k}{k-1},
$$
that is,
\begin{eqnarray}\label{eqtriangle}2s'_{k-1,g-3}= s'_{k,g}-s'_{k,g-1},
\end{eqnarray}
with the convention that $s'_{k,g}=0$ for $k<0$.
It follows by definition of $b'_{l,k}$ that
$$2s'_{k-1,g-3}=\sum_{l=0}^{k}b'_{l,k} s'_{k-l,g-1}-\sum_{l=0}^{k}b'_{l,k} s'_{k-l,g-2}=\sum_{l=0}^{k}b'_{l,k}(s'_{k-l,g-1}- s'_{k-l,g-2}),$$
which gives, by applying (\ref{eqtriangle}) again to each term
 in the right hand side:
$$2s'_{k-1,g-3}=2\sum_{l=0}^kb'_{l,k}s'_{k-l-1,g-4}=2\sum_{l=0}^{k-1}b'_{l,k}s'_{k-l-1,g-4}.$$
By definition of $b'_{l,k-1}$, this provides $b'_{l,k}=b'_{l,k-1}$.
\end{proof}
\section{Further geometric vanishing}
We discuss in this section similar geometric vanishing results for the Segre classes on the blow-up of a $K3$ surface at one point.
The setting is thus the following: $S$ is a $K3$ surface with ${\rm Pic}\,S=\mathbb{Z}L,\,L^2=2g-2$,
and $x\in S$ is a point. The surface $\tau:\widetilde{S}\rightarrow S$ is the blow-up of $S$ at $x$ with exceptional curve $E$,
and $H:=\tau^*L(-l E)\in {\rm Pic}\,\widetilde{S}$ for some positive  integer $l$.
Our main goal is to discuss the analogue of Lemma \ref{lelaz}
in this context.
Note that, when $H$ is very ample, the curve $E$ has degree $l$ in the embedding given
by $|H|$, so that the vector bundle $\mathcal{H}_{[k]}$ can be generated by sections only when $k\leq l+1$.

To start with, we have:
\begin{prop}\label{propourEsansZ} Let $S$ be a $K3$ surface with Picard group generated by
$L$, $L^2=2g-2$. Let $\tau:\widetilde{S}\rightarrow S$ be the blow-up at a point $x\in S$. Then, denoting $H=\tau^*L(-lE)$,
if
\begin{eqnarray}\label{eqinegstilde}4+2g>(l+1)^2,
 \end{eqnarray}  one has $H^1(\widetilde{S},H)=0$. It follows that
$h^0(\widetilde{S},H)=g+1-\frac{l(l+1)}{2}$.
\end{prop}
\begin{proof} We argue by contradiction. The proof follows
 Reider's \cite{reider} and Lazarsfeld's \cite{laz} methods.
Assume $H^1(\widetilde{S},H)\not=0$. Then, by Serre duality,
${\rm Ext}^1(H,\mathcal{O}_{\widetilde{S}}(E))\not=0$, which provides  a rank $2$ vector bundle
$\mathcal{E}$ on $\widetilde{S}$ which fits in an exact sequence
\begin{eqnarray}
\label{eqexpourdefE} 0\rightarrow \tau^*L^{-1}((l+1) E)\rightarrow \mathcal{E}\rightarrow \mathcal{O}_{\widetilde{S}}
\rightarrow 0.
\end{eqnarray}
The fact that the extension class of (\ref{eqexpourdefE})  is not trivial translates into
$h^0(\widetilde{S},\mathcal{E})=0$.
We have $c_2(\mathcal{E})=0$ and $c_1(\mathcal{E})^2=2g-2-(l+1)^2$, so that
(\ref{eqinegstilde}) gives the inequality
$$\chi({End}\,\mathcal{E})=8+c_1(\mathcal{E})^2-4c_2(\mathcal{E})>2.$$
It follows that $h^0(\widetilde{S},{End}\,\mathcal{E})+h^0(\widetilde{S},{End}\,\mathcal{E}(E))>2$,
hence $h^0(\widetilde{S},{End}\,\mathcal{E}(E))>1$. Thus there exists a $\phi\in {\rm Hom}\,(\mathcal{E},\mathcal{E}(E))$ which is not proportional to the identity. The characteristic polynomial of $\phi$ has its trace in $H^0(\widetilde{S},\mathcal{O}_{\widetilde{S}}(E))=H^0(\widetilde{S},\mathcal{O}_{\widetilde{S}})$ and determinant
in $H^0(\widetilde{S},\mathcal{O}_{\widetilde{S}}(2E))=H^0(\widetilde{S},\mathcal{O}_{\widetilde{S}})$.
It is thus a polynomial with coefficients in $\mathbb{C}$ and has a rood $\lambda$.
Replacing $\phi$ by $\phi-\lambda Id_\mathcal{E}$ (where we see $Id_\mathcal{E}$ as an element of
${\rm Hom}\,(\mathcal{E},\mathcal{E}(E))$),
we can in fact assume that $\phi$ is generically of rank $1$.
Let $A={\rm Ker}\,\phi\subset \mathcal{E}$. We have $A= \tau^*L^\alpha(\beta E)$ and
$\mathcal{E}$ fits in an exact sequence
\begin{eqnarray}\label{egexcatphiE}
0\rightarrow A\rightarrow \mathcal{E}\rightarrow B\otimes \mathcal{I}_W\rightarrow 0,
\end{eqnarray}
where $B$ is the line bundle $\tau^*L^{-1-\alpha}((l+1-\beta) E)$.
As $B={\rm Im}\,\phi$, we have $B\hookrightarrow\mathcal{E}(E)$.
From the exact sequence (\ref{eqexpourdefE}), we immediately conclude that
$\alpha\leq 0$ and $(-1-\alpha)\leq 0$, so that $\alpha=0$ or $\alpha=-1$.

Assume first $\alpha=0$. Then as $h^0(\widetilde{S},\mathcal{E})=0$, we conclude that
$\beta<0$, hence $l+1-\beta>0$. Then (\ref{egexcatphiE}) gives $$c_2(\mathcal{E})=A\cdot B+{\rm deg}\,W\geq -\beta(l+1-\beta)>0,$$
which is a contradiction.

In the remaining case $\alpha=-1$, we conclude that $B=\mathcal{O}_{\widetilde{S}}((l+1-\beta) E)$, so that we have
a nonzero morphism $\mathcal{O}_{\widetilde{S}}((l-\beta)E)\rightarrow \mathcal{E}$. This provides
a line bundle $A'\subset \mathcal{E}$ defined as the saturation of the image
 of this morphims, and we know that $A' =\tau^*L^{\alpha'}(\beta' E)$ with $\alpha'\geq 0$.
 We can then apply the previous argument with $A$ replaced by $A'$, getting a contradiction.
\end{proof}
Pushing forward the arguments above, we   now prove the following result:
\begin{theo}\label{theoexconj} Let $S$ be a general $K3$ surface with Picard group generated by $L$, and $x\in S$ a general point. Then for $k\geq 2$, $H=\tau^*L(-lE)$  is $k$-ample for
$k=l$  or $k=l+1$, and  $g-\frac{l(l+1)}{2}=3k-2$.
\end{theo}
\begin{rema}\label{rematardive}{ When $g-\frac{l(l+1)}{2}=3k-2$, with $k=l$  or $k=l+1$, one has for $l>0$
$$4+2g=l(l+1)+6k\geq (l+7)l>(l+1)^2$$
so that Proposition \ref{propourEsansZ} applies, which gives
$H^1(\widetilde{S},H)=0$ and  $h^0(\widetilde{S},H)=g+1-\frac{l(l+1)}{2}=3k-1$.
}
\end{rema}
\begin{proof}[Proof of Theorem \ref{theoexconj}]  With the assumptions of Theorem \ref{theoexconj}, assume
$H$ is not $k$-ample. Therefore there exists a $0$-dimensional subscheme
$Z\subset \widetilde{S}$ of length $k$ such that
$H^1(\widetilde{S},H\otimes\mathcal{I}_Z)\not=0$. Using
the duality $H^1(\widetilde{S},H\otimes\mathcal{I}_Z)^*={\rm Ext}^1(\mathcal{I}_Z,-H+E)$,
this provides us with
a rank $2$ torsion free sheaf $\mathcal{E}$ on $\widetilde{S}$ fitting in an exact sequence
\begin{eqnarray}\label{eqexdefEZ}0\rightarrow \tau^*L^{-1}((l+1)E)\rightarrow\mathcal{E}\rightarrow
\mathcal{I}_Z\rightarrow 0.
\end{eqnarray}
The numerical invariants of $\mathcal{E}$ are
given by
$$c_2(\mathcal{E})=k,\,c_1(\mathcal{E})^2=2g-2-(l+1)^2,$$
from which we conclude that
$$\chi(\mathcal{E},\mathcal{E})=8+2g-2-(l+1)^2-4k,$$
hence
\begin{eqnarray} \label{eqsuppourend} h^0(End\,\mathcal{E})+h^0(End\,\mathcal{E}(E))\geq 8+2g-2-(l+1)^2-4k.
\end{eqnarray}
By assumption,  $g-\frac{l(l+1)}{2}=3k-2$, so $2g-2-(l+1)^2=6k-6-(l+1)$ and (\ref{eqsuppourend}) gives
$$2h^0(End\,\mathcal{E}(E))\geq 2+2k-(l+1),$$
hence $2h^0(End\,\mathcal{E}(E))>2$ because $k\geq2$ and $k=l$ or $k=l+1$. Thus there exists a morphism
$$\phi:\mathcal{E}\rightarrow \mathcal{E}(E)$$
which is not proportional to the identity. As before,
we can even assume that $\phi$ is generically of rank $1$. One difference with the previous situation
is the fact that $\mathcal{E}$ is not necessarily locally free, and furthermore $c_2(\mathcal{E})\not=0$.
The kernel of $\phi$ and its image are torsion free of rank $1$,
hence are of the form
$A\otimes \mathcal{I}_W,\,B\otimes \mathcal{I}_{W'}$ for some line bundles
$A,\,B$
on $\widetilde{S}$ which are of the form
$$A=\tau^*L^\alpha(\beta E),\,\,B=\tau^*L^{-1-\alpha}((l+1-\beta)E).$$
As before, we must have $\alpha\leq 0$ and $-1-\alpha\leq 0$ because $B$ injects into
$\mathcal{E}(E)$.
Hence we conclude that $\alpha=0$ or $\alpha=-1$.

(i) If $\alpha=0$, then we have a nonzero morphism
$\mathcal{O}(\beta E)\otimes\mathcal{I}_W\rightarrow \mathcal{I}_Z$. It follows that
$\beta\leq 0$. If $\beta=0$,
this says that $\mathcal{I}_W\subset \mathcal{I}_Z$ and that the
extension class of (\ref{eqexdefEZ}) vanishes in
${\rm Ext}^1(\mathcal{I}_W,\tau^*L^{-1}((l+1)E))$. But the restriction map
$${\rm Ext}^1(\mathcal{I}_Z,\tau^*L^{-1}((l+1)E))\rightarrow {\rm Ext}^1(\mathcal{I}_W,\tau^*L^{-1}((l+1)E))$$
is injective as it is  dual to the map
$H^1(\widetilde{S},\mathcal{I}_W(H))\rightarrow H^1(\widetilde{S},\mathcal{I}_Z(H))$ which is surjective. Indeed, the spaces are respective quotients of $H^0(H_{\mid W})$, $H^0(H_{\mid Z})$ by
Proposition \ref{propourEsansZ} which applies in our case as noted in Remark
\ref{rematardive}. So we conclude that $\beta<0$. We now compute
$c_2(\mathcal{E})$ using the exact sequence
$$0\rightarrow A\otimes \mathcal{I}_W\rightarrow \mathcal{E}\rightarrow B\otimes \mathcal{I}_{W'}\rightarrow 0,$$
with $A=\mathcal{O}(\beta E)$, $B=\tau^*L^{-1}((l+1-\beta) E)$.
This gives $$c_2(\mathcal{E})={\rm deg}\,W+{\rm deg}\,W'-\beta(l+1-\beta)\geq -\beta(l+1-\beta)\geq l+2.$$
This contradicts $c_2(\mathcal{E})=k\leq l+1$.

(ii) If $\alpha=-1$, then we use instead the inclusion
$B\otimes \mathcal{I}_{W'}\subset \mathcal{E}(E)$, with $B=\mathcal{O}((l+1-\beta)E)$ and argue exactly as before.
\end{proof}
 We  deduce  the following Corollary \ref{coroannueclat} concerning the numbers
 $s_{k}(d,\pi,\kappa,e)$ (we adopt here Lehn's notation \cite{lehn})
defined as the top Segre class of $\mathcal{H}_{[k]}$ for a
pair $(\Sigma,H)$ where $\Sigma$ is a smooth compact surface, and
$$d=H^2,\,\pi=H\cdot c_1(K_\Sigma),\,\kappa=c_1(\Sigma)^2,\,e=c_2(\Sigma).$$

\begin{coro}\label{coroannueclat} (Cf. Theorem \ref{theoannuintro}.)
One has the following vanishing for $s_{k}(d,\pi,-1,25)$
\begin{eqnarray}\label{eqvanconj} s_{k}(7(k-1),k-1,-1,25)=0,\,\,s_{k}(7(k-1)+1,k,-1,25)=0
\end{eqnarray}
for $k\geq 2$.
\end{coro}
\begin{proof} Take for  $\Sigma$  the blow-up of a $K3$ surface  at a point so
$$\kappa=-1,\,e=25.$$
Furthermore, assuming ${\rm Pic}\,S=\mathbb{Z}L$ with
$L^2=2g-2$, and letting $H=\tau^*L(-lE)$ as above, we have
\begin{eqnarray} d=H^2=2g-2-l^2,\,\pi=H\cdot c_1(K_\Sigma)=l.
\label{eqnombres1}
\end{eqnarray}
 We  consider the cases where
 \begin{eqnarray}g-\frac{l(l+1)}{2}=3k-2\label{eqnombres22}
\end{eqnarray}
 with (i) $k=l+1$ or (ii) $k=l$.

  Using (\ref{eqnombres1}), (\ref{eqnombres22}) gives
 in  case (i), $d=7(k-1),\,\pi=k-1$ and in case (ii), $d=7(k-1)+1,\pi=k$, so we are exactly computing $s_{k}(7(k-1),k-1,-1,25)=0$ in case (i) and $s_{k}(7(k-1)+1,k,-1,25)$ in case (ii).
   Remark \ref{rematardive} says that assuming  (\ref{eqnombres22}),
 $$H^1(\widetilde{S},H)=0,\,h^0(\widetilde{S},H)=3k-1$$
 in cases (i) and (ii).
 Theorem \ref{theoexconj} says that under the same assumption, $H$ is $k$-ample on $\widetilde{S}$.
  Lemma \ref{lefact} thus applies and gives $s_{2k}(\mathcal{H}_{[k]})=0$ in both cases, which is exactly (\ref{eqvanconj}).
 \end{proof}
 \begin{rema}{\rm Lehn gives in \cite[Section 4]{lehn} the explicit polynomial formulas for
 $2! s_2,\ldots,\,5!s_5$ as polynomial functions of $d,\,\pi,\,\kappa,\,e$ with huge integral coefficients.
 For example
 \begin{eqnarray}
 \label{eqlehn5}
  5! · s_5 = d^5 - 100d^4 + d^3(3740 + 10e - 50\pi - 10\kappa)\\
  \nonumber
-d^2(62000 - 3420\pi + 700e - 860\kappa) + d(384384 + 15e^2 \\
\nonumber
+15960e - 30e\kappa - 150\pi e + 15\kappa^2 + 150\kappa\pi - 75610\pi
-24340\kappa + 375\pi^2)\\
\nonumber
 - 400e^2 - 117120e + 3920\pi e + 960\kappa e
+226560\kappa - 4720\kappa\pi
\\
\nonumber- 560\kappa^2 + 530880\pi - 9600\pi^2
 \end{eqnarray}
 It is  pleasant to check the vanishing statements  (\ref{eqvanconj}) for $k=2,\ldots,\,5$ using these formulas. For $k=5$, one just has to plug-in the values
 $e=25$, $\kappa=-1$, $d=28$  and $ \pi=4$, or $e=25$, $\kappa=-1$, $d=29$  and $\pi=5$
 in (\ref{eqlehn5}).
 }
 \end{rema}
 We conclude this note by showing that all the Segre numbers are formally determined
 by the above  results and formula (\ref{eqsumprod}).
\begin{prop} \label{prodeterm} The vanishings (\ref{eqvanconj}) together with the data of the numbers
$ s_{k}(d,0,0,24)$   and $ s_{k}(d,0,0,0)$  determine all numbers $s_{k}(d,\pi,\kappa,e)$.
\end{prop}
 Note that $ s_{k}(d,0,0,24)$ is for $d=2g-2$ the number $s_{k,g}$ of the introduction, and these numbers  are given by Marian-Oprea-Pandharipande's Theorem \ref{theopandha}. The numbers $ s_{k}(d,0,0,0)$ correspond for $d$ even to the Segre classes of tautological sheaves on  Hilbert schemes of abelian surfaces
 equipped with a line bundle of self-intersection $d$. They are fully determined, by multiplicativity,
 by the case of self-intersection $2$, where one gets the numbers $b'_k$ appearing in our  proof of Theorem
 \ref{theopandha}.
\begin{proof}[Proof of Proposition \ref{prodeterm}] According to
\cite{lehn}, \cite{egl}, and as follows from
(\ref{eqsumprod}), the generating  series
$$s(z)=\sum_k s_k(d,\pi,\kappa,e) z^k$$
is of the form \begin{eqnarray}\label{eqformmults} s(z)=A(z)^d B(z)^e C(z)^\pi D(z)^\kappa,
 \end{eqnarray}for power series
$A,\,B,\,C,\,D$ with $0$-th order coefficient equal to  $1$.
 Theorem \ref{theopandha} determines the series
$A(z)$ and $B(z)$. We thus only have to determine $C(z)$ and $D(z)$. The degree $1$ coefficients
of the power series
$C(z),\,D(z)$ are immediate to compute as $s_1=d$. We now assume that the coefficients of the power series
$C(z)$ and $D(z)$ are computed up to degree $k-1$.
The degree $k$ coefficient of
$s(z)=A(z)^d B(z)^e C(z)^\pi D(z)^\kappa$ is of the form
$\pi C_k+\kappa D_k+ \nu$ where $\nu$ is determined by $d,\,e,\,\pi,\,\kappa$, the coefficients of $A$ and $B$, and the
coefficients of order $\leq k-1$ of $C$ and $D$.
The vanishings
(\ref{eqvanconj}) thus give the equations
$$ 0=(k-1)C_k-D_k+\nu,\,\,0=k C_k-D_k+\nu',$$
which obviously determines $C_k$ and $D_k$ as functions of $\nu$ and $\nu'$.
\end{proof}
We finally prove Corollary \ref{coropourlehnintro} of the introduction.
\begin{proof}[Proof of Corollary \ref{coropourlehnintro}] Let $f_{d,\pi,\kappa,e}(z)$ be the Lehn function introduced in Conjecture \ref{conjlehn}.
As Lehn's conjecture is proved by \cite{MOP} for $\pi=\kappa=0$ (the $K$-trivial case), the
coefficients $f_{k,d,\pi,\kappa,e}$ of the Taylor expansion of $f_{d,\pi,\kappa,e}$ in $z$ (not $w$) are the Segre numbers
$s_{k,d,0,0,e}$ when $\pi=0,\,\kappa=0$. If furthermore they satisfy the vanishings
$f_{k,d,\pi,\kappa,e}=0$ for $e=25,\,\kappa=-1$ and $d=7(k-1),\,\pi=k-1$ or $d=7(k-1)+1,\,\pi=k$, the proof of
Proposition \ref{prodeterm} shows that $f_{k,d,\pi,\kappa,e}=s_{k,d,\pi,\kappa,e}$ for
all $k,\,d,\,\pi,\,\kappa,\,e$ as, by definition, $f$ has the same multiplicative form (\ref{eqformmults}) as $s$.
\end{proof}

Coll\`{e}ge de France, 3 rue d'Ulm 75005 Paris

claire.voisin@imj-prg.fr
    \end{document}